\documentclass[a4paper,twoside,12pt]{article}
\usepackage{amssymb,amsfonts,amsmath,amsthm,latexsym}
\usepackage[pdftex,bookmarks,colorlinks=false]{hyperref}
\usepackage[hmargin=1.2in,vmargin=1.2in]{geometry}
\usepackage{graphicx}
\usepackage[auth-lg,affil-it]{authblk}
\usepackage{tabularx}

\newtheorem{thm}{Theorem}[section]
\newtheorem{cor}[thm]{Corollary}
\newtheorem{defn}[thm]{Definition}

\newtheorem{prob}[thm]{Problem}
\newtheorem{prop}[thm]{Proposition}

\numberwithin{equation}{section}
\def\ni{\noindent}
\def\N{\mathbb{N}}

\title{\textbf{\sc Some New Results on the Curling Number of Graphs}}

\author{N. K. Sudev$^{\dagger}$, C. Susanth$^{\ddagger}$}
\affil{\small Centre for Studies in Discrete Mathematics \\ Vidya Academy of Science and Technology \\Thalakkottukara, Thrissur, Kerala, India.\\ {\tt sudevnk@gmail.com$^{\dagger}$, susanth\_c@yahoo.com$^{\ddagger}$.}}

\author{K. P. Chithra}
\affil{\small Naduvath Mana, Nandikkara \\ Thrissur, India.\\ {\tt chithrasudev@gmail.com}}

\author{Johan Kok}
\affil{\small Tshwane Metropolitan Police Department \\ City of Tshwane, South Africa. \\ {\tt kokkiek2@tshwane.gov.za}}

\author{Sunny Joseph Kalayathankal}
\affil{\small Department of Mathematics\\ Kuriakose Elias College\\ Kottayam, Kerala, India.\\ {\tt sunnyjoseph2014@yahoo.com}}

\date{}

\begin{document}
\maketitle

%\newpage 

\begin{abstract}
Let $S=S_1S_2S_3\ldots S_n$ be a finite string. Write $S$ in the form $XYY\ldots Y=XY^k$, consisting of a prefix $X$ (which may be empty), followed by $k$ copies of a non-empty string $Y$.  Then, the greatest value of this integer $k$ is called the curling number of $S$ and is denoted by $cn(S)$. Let the degree sequence of the graph $G$ be written as a string of identity curling subsequences say, $X^{k_1}_1\circ X^{k_2}_2\circ X^{k_3}_3 \ldots \circ X^{k_l}_l$. The compound curling number of $G$, denoted $cn^c(G)$ is defined to be, $cn^n(G) = \prod\limits^{l}_{i=1}k_i$. In this paper, we discuss the curling number and compound curling number of  certain products of graphs. 
\end{abstract}

\ni {\bf Keywords:} Curling number of a graph, compound curling number of a graph.

\ni {\footnotesize \textbf{AMS Classification Numbers:} 05C07, 05C76, 11B83.}
 
\section{Introduction}

For general notation and concepts in graph theory, we refer to \cite{BM1,CL1,GY1,FH,DBW}. For different types of graphs products, we further refer to \cite{HIS,IK}.  All graphs mentioned in this paper are simple, non-trivial, connected and finite graphs unless mentioned otherwise. 

The notion of \textit{curling number} of a sequence of integers is introduced in \cite{CLSW} as follows. 

\begin{defn}{\rm 
\cite{CLSW} Let $S=S_1S_2S_3\ldots S_n$ be a finite string. Write $S$ in the form $XYY\ldots Y=XY^k$, consisting of a prefix $X$ (which may be empty), followed by $k$ copies of a non-empty string $Y$. This can be done in several ways. Pick one with the greatest value of $k$ . Then, this integer $k$ is called the {\em curling number} of $S$ and is denoted by $cn(S)$.}
\end{defn} 

If we partition the sequence $S$ into two subsequences, say $X,Y$, we can write $S$ as the string $S =X\circ Y$. More useful and important studies on the curling number of integer sequences have been done in \cite{CLSW,CS1,FW1,NJAS}.

Extending the concepts of curling number of sequences mentioned in the above studies in to the degree sequences of graphs, the notion of the curling number of a graph $G$ has been introduced in \cite{KSC} as follows.

\begin{defn}{\rm 
\cite{KSC} Given a finite non-empty graph $G$ with the degree sequence $S=(a_1, a_2, a_3, \ldots, a_n)$, $a_i \in \N_0$.  This degree sequence $S$  can be written as a string of subsequences $S = X_1^{k_1}\circ X_2^{k_2}\circ \ldots \circ X_l^{k_l}$. Then, the curling number of $G$, denoted by $cn(G)$, is defined to be $cn(G)=\max\{k_i\}$, where $1 \le i \le l$.}
\end{defn}

The notion of a curling subsequence of a a given graph has been introduced in \cite{KSC}  as follows.

\begin{defn}{\rm 
\cite{KSC}  A \textit{curling subsequence} of a simple connected graph $G$ is defined to be a maximal subsequence $C$ of the well-arranged degree sequence of $G$ such that $cn(C) = \max\{cn(S_0)\}$ for all possible initial subsequences $S_0$. } 
\end{defn}

The concept of the compound curling number of a graph $G$ has also been introduced in \cite{KSC} as follows.

\begin{defn}\label{Def-3.1}{\rm 
\cite{KSC} Let the degree sequence of the graph $G$ be written as a string of identity curling subsequences say, $X^{k_1}_1\circ X^{k_2}_2\circ X^{k_3}_3 \ldots \circ X^{k_l}_l$. The \textit{compound curling number} of $G$, denoted $cn^c(G)$, is defined to be $cn^c(G) = \prod\limits^{l}_{i=1}k_i$, where $1 \le i \le l$.}
\end{defn}

The following is an important proposition on the curling number and compound curling number of regular graphs and is a relevant result in our present study. 

\begin{prop}\label{P-CCN}
{\rm \cite{KSC}} The curling number and the compound curling number of a regular graph are the same and are equal to the order of that graph.
\end{prop}

The curling number and compound curling number of certain fundamental standard graphs have been determined in \cite{KSC} and the major results are listed in the following table.
\begin{table}[h!]
\begin{center}
\begin{tabular}{|c|l|c|c|}
\hline
\textbf{No.} & \textbf{Graph} & $cn$ & $cn^c$ \\
\hline
$1$ & Complete Graph $K_n, n\geq 1$ & $n$ & $n$ \\
\hline
$2$ & Complete Bipartite Graph $K_{m,n}, m\neq n$ & max $\{m,n\}$ & $mn$ \\
\hline
$3$ & Complete Bipartite Graph $K_{n,n}$ & $2n$ & $n^2$ \\
\hline
$4$ & Path Graph $P_n, n\geq 4$ & $n-2$  & $2(n-2)$ \\
\hline
$5$ & Cycle $C_n$ & $n$ & $n$ \\
\hline
\end{tabular}
\end{center}
\end{table}

The curling number and compound curling number of the powers of certain graph classes have been studied in \cite{SSSCK}. Motivated from these studies, in this paper, we discuss the curling number and compound curling number of certain operations and products of graphs.

\section{Curling Number of Join of Graphs}

Let $G_1(V_1,E_1)$ and $G_2(V_2,E_2)$ be two graphs. Then, their {\em join}, denoted by $G_1+G_2$, is the graph whose vertex set is $V_1\cup V_2$ and edge set is $E_1\cup E_2\cup E_{ij}$, where $E_{ij}=\{u_iv_j:u_i\in G_1,v_j\in G_2\}$. The following theorem discusses the range of curling number of the join of two graphs.

\begin{thm}\label{T-CNJG1}
Let $G_1$ and $G_2$ be two graphs. Then, $\max\{cn(G_1), cn(G_2)\}\le cn(G_1+G_2)\le cn(G_1)+cn(G_2)$.
\end{thm}
\begin{proof}
Let $V_1=\{v_1,v_2,v_3\ldots,v_n\}$ and $V_2=\{u_1,u_2,u_3\ldots,u_m\}$ be the vertex sets of the given graphs $G_1$ and $G_2$ respectively. Let $G=G_1+G_2$. Then, for $1\le i\le n$, every vertex $v_i$ of $G_1$ has the condition $d_G(v_i)= d_{G_1}(v_i)+m$ and for $1\le j\le m$, every vertex $u_j$ of $G_2$ has the condition $d_G(u_j)= d_{G_2}(u_j)+n$. 

If $d_{G_1}(v_i)+m\ne d_{G_2}(u_j)+n$, then the curling number of the subgraph $G-E(G_2)$ of $G$ is the same as $cn(G_1)$ and the curling number of the subgraph $G-E(G_1)$ of $G$ is the same as $cn(G_2)$. Hence, in this case, the curling number of $G_1+G_2$ is $\max\{cn(G_1),cn(G_2)\}$.

Assume that $d_{G_1}(v_i)+m=d_{G_2}(u_j)+n=q$, a constant, for some vertices $v_i\in V_1$ and for some vertices $u_j\in V_2$. The maximum possible number of vertices in $G_1$ which satisfy this equality is $cn(G_1)$ and the maximum possible number of vertices in $G_2$ which satisfy the above equality is $cn(G_2)$. Therefore, the maximum possible value of $cn(G_1+G_2)$ is $cn(G_1)+cn(G_2)$.

Hence, we have $\max\{cn(G_1), cn(G_2)\}\le cn(G_1+G_2)\le cn(G_1)+cn(G_2)$. This completes the proof.
\end{proof}

Determining the compound curling number of the join of two arbitrary graphs is a complex problem because of the uncertainty in their degree sequences. On the other hand, it is easier to determine the compound curling number of the join of two graphs having known or predictable pattern of degree sequences. Hence, we discuss the curling number and compound curling number of regular graphs in the following results.

\begin{prop}
If $G_1(V_1,E_1)$ and $G_2(V_2,E_2)$ are two regular graphs, then 
$$cn(G_1+G_2)=
\begin{cases}
|V_1|+|V_2|, & \text{if}~ ~ d_{v\in V_1}(v)+|V_2|=d_{u\in V_2}(u)+|V_1|,\\
\max\{|V_1|,|V_2|\},  & \text{otherwise}.
\end{cases}$$
\end{prop}
\begin{proof}
If $d_{v\in V_1}(v)+|V_2|=d_{u\in V_2}(u)+|V_1|$, then $G_1+G_2$ is a regular graph on $|V_1|+|V_2|$ vertices and hence by Proposition \ref{P-CCN}, $cn(G_1+G_2)=|V_1|+|V_2|$.

If $d_{v\in V_1}(v)+|V_2|\ne d_{u\in V_2}(u)+|V_1|$, as proved in Theorem \ref{T-CNJG1}, we have $cn(G_1+G_2-E(G_2))=cn(G_1)=|V_1|$ and $cn(G_1+G_2-E(G_1))=cn(G_2)=|V_2|$ and hence $cn(G_1+G_2)=\max\{|V_1|,|V_2|\}$. This completes the proof.
\end{proof}

The following result describes the compound curling number of the join of two regular graphs.

\begin{prop}
For two regular graphs $G_1(V_1,E_1)$ and $G_2(V_2,E_2)$, we have 
$$ cn^c(G_1+G_2)=
\begin{cases}
|V_1|+|V_2|, & \text{if}~ ~ d_{v\in V_1}(v)+|V_2|=d_{u\in V_2}(u)+|V_1|,\\
|V_1|\,|V_2|,  & \text{otherwise}.
\end{cases}$$
\end{prop}
\begin{proof}
If $d_{v\in V_1}(v)+|V_2|=d_{u\in V_2}(u)+|V_1|$, then $G_1+G_2$ is a regular graph on $|V_1|+|V_2|$ vertices and hence by Proposition \ref{P-CCN}, $cn^c(G_1+G_2)=cn(G_1+G_2)=|V_1|+|V_2|$.

If $d_{v\in V_1}(v)+|V_2|\ne d_{u\in V_2}(u)+|V_1|$, then $cn^c(G_1+G_2-E(G_2))=cn(G_1+G_2-E(G_2))=cn(G_1)=cn^c(G_1)=|V_1|$ and $cn^c(G_1+G_2-E(G_1))=cn(G_1+G_2-E(G_1))=cn(G_2)=cn^c(G_2)=|V_2|$ and hence $cn^c(G_1+G_2)=|V_1|\,|V_2|$. 
\end{proof}

\section{Curling Number of the Cartesian Products of Two Graphs}

Let $G_1(V_1,E_1)$ and $G_2(V_2,E_2)$ be the two given graphs. The {\em Cartesian product} of $G_1$ and $G_2$ (see \cite{HIS,FH}), denoted by $G_1\Box G_2$, is the graph with vertex set $V_1\times V_2$, such that two points $u=(v_i, u_j)$ and $v=(v_k,u_l)$ in $V_1\times V_2$ are adjacent in $G_1\Box G_2$ whenever [$v_i=v_k$ and $u_j$ is adjacent to $u_l$] or [$u_j=u_l$ and $v_i$ is adjacent to $v_k$]. 

Let $G_1(V_1,E_1)$ and $G_2(V_2,E_2)$ be two non-empty graphs. Let us represent the vertex $(v_i,u_j)$ by $v_{ij}$. Then, for their Cartesian product $G_1\Box G_2$, let $C_k=\{v_{ij}\in V_1\times V_2: d(v_i)+d(v_j)=k\}$, where $\delta(G_1)+\delta(G_2)\le k\le \Delta(G_1)+\Delta(G_2)$. The set $C_k$ is called the compatible class of the integer $k$ in $G_1\Box G_2$ (see \cite{GS0} for the definition of compatible classes). Then, the curling number of $G_1\Box G_2$ is the cardinality of the maximal compatibility class in $G_1\Box G_2$. That is $cn(G)=\{|C_k|\}$.  

The maximal compatible class of the Cartesian product of two graph depends upon the lengths of identity subsequences of the degree sequences of both graphs. Hence, determining the curling number and compound curling number of two graphs is possible only when pattern of degree sequences is known or predictable. The following proposition discusses the curling number and compound curling number of the Cartesian product of two regular graphs.

\begin{prop}\label{P-CNP1a}
If $G_1(V_1,E_1)$ and $G_2(V_2,E_2)$ are two regular graphs, then we have $cn(G_1\Box G_2)=cn^c(G_1\Box G_2)= |V_1|\,|V_2|$.
\end{prop}
\begin{proof}
If $G_1$ is an $r_1$-regular graph on $n_1$ vertices and $G_2$ is an $r_2$-regular graph on $n_2$ vertices, then $G_1\Box G_2$ is an $(r_1+r_2)$-regular graph on $n_1n_2$ vertices. Therefore, the degree sequence of $G_1\Box G_2$ is given by $S=(r_1+r_2)^{n_1n_2}=\epsilon^1 \circ (r_1+r_2)^{n_1n_2}$, where $\epsilon$ is the null subsequence of $S$. Therefore, $cn(G_1\Box G_2)=n_1n_2$. Also, $cn^c(G_1\Box G_2)=1\cdot (n_1n_2)=cn(G_1\Box G_2)$. This completes the proof.
\end{proof}

The curling number and compound curling number of certain standard graphs, which are the Cartesian product of certain standard graphs, are determined in the following results.

First, we consider a graph $G$ which is the Cartesian product of two paths. Note that the graph $P_2\Box P_2$ is isomorphic to the cycle $C_4$. Since the curling number of cycles have already been determined, we need consider the case $m=n=2$. For $m\ge 3, n=2$, the graph $P_m\Box P_2$ is known as a {\em ladder graph} and is usually denoted by $L_m$. The following result discusses the curling number and compound curling number of a ladder graph.

\begin{prop}
For $m\ge 3$, the curling number of a ladder graph $L_m$ is $cn(L_m)=\max\{4,2m-4\}$ and its compound curling number is $cn^c(L_m)=8(m-2)$.
\end{prop}
\begin{proof}
In $L_n$, there are $4$ vertices of degree $2$ and $2(m-2)$ vertices of degree $3$. Therefore, is $cn(L_m)=\max\{4,2(m-2)\}$ and $cn^c(L_m)=4\cdot 2(m-2)=8(m-2)$.
\end{proof}

From the above proposition, it is clear that the curling number of a ladder graph $L_m$ is $4$ for $m\le 4$ and is $2m-4$ for $m\ge 4$. 

The following proposition discusses the curling number and compound curling number of a planar grid graph $G=P_m\Box P_n$, for $m,n\ge 3$.

\begin{prop}\label{prop-3.3}
For $m,n\ge 3$, a planar grid graph $G=P_m\Box P_n$,  we have 
$cn(G)=\max\{2(m+n-4),(m-2)(n-2)\}$ and $cn^c(G)=8((m-2)^2(n-2)+(m-2)(n-2)^2)$.
\end{prop}
\begin{proof}
Note that in $P_m\Box P_n$, there are $(m-2)(n-2)$ vertices of degree $4$, $2(m+n-4)$ vertices of degree $3$ and $4$ vertices of degree $2$. Hence, the degree sequence $S$ of the planar grid graph $P_m\Box P_n$ is given by $S=(4)^{(m-2)(n-2)}\circ (3)^{2(m+n-4)}\circ (2)^4$. Therefore, we have $cn(P_m\Box P_n)=\max\{2(m+n-4), (m-2)(n-2)\}$. Moreover, $cn^c(P_m\Box P_n)=8(m-2)(n-2)(m+n-4)=8((m-2)^2(n-2)+(m-2)(n-2)^2)$.
\end{proof}

In view of Proposition \ref{prop-3.3}, we observe that the curling number of a planar grid graph $P_m\Box P_n$ is $(m-2)(n-2)$ when either $m=n=6$ or $m,n\ge 5, m+2\ge 13$. On all other cases, $cn(P_m\Box P_n)=2(m+n-4)$.

Next, we discuss the curling number of the prism graphs $C_m\Box P_n$. Note that the prism graph $C_m\Box P_2$ is a $3$-regular graph and hence by Proposition \ref{P-CNP1a}, its curling number and the compound curling number are the same and is $2m$. Similarly, for $n=3$, the prism graph $C_m\Box P_3$ consists of $2m$ vertices of degree $4$. Hence, $cn(C_m\Box P_3)=2m$ and $cn^c(C_m\Box P_3)=2m^2$.

Hence, we need to consider the paths $P_n$ with $n\ge 4$ only for further studies on the curling number of prism graphs. In the following result, we discuss the curling number and compound curling number of prism graphs. 

\begin{prop}
For $n\ge 4$, the curling number of a prism graph $G=C_m\Box P_n$ is $cn(G)=m(n-2)$ and $cn^c(G)=2m^2(n-2)$.
\end{prop}
\begin{proof}
In the prism graph $C_m\Box P_n, n\ge 4$, there are $m(n-2)$ vertices of degree $4$ and $2m$ vertices of degree $3$. Hence, the degree sequence $S$ of $C_m\Box P_n$ is given by $S=(4)^{m(n-2)}\circ (3)^{2m}$. Therefore, we have $cn(C_m\Box P_n)=m(n-2)$ and $cn^c(C_m\Box P_n)=2m^2(n-2)$.
\end{proof}

The curling number and compound curling number of a torus grid graph $C_m\Box C_n$ are determined in following proposition.

\begin{prop}
The curling number and the compound curling number of the torus grid graph $C_m\Box C_n$ is $mn$.
\end{prop}
\begin{proof}
The torus grid graph $C_m\Box C_n$ is a $4$-regular graph and hence, by Proposition \ref{P-CNP1a}, we have $cn(C_m\Box C_n)=cn^c(C_m\Box C_n)=mn$. 
\end{proof}

A \textit{stacked book graph}, denoted by $B_{m,n}$, is a graph defined by $B_{m,n}=K_{1,m}\Box P_n$. The graph $B_m=K_{1,m}\Box P_2$ is called a \textit{book graph}. 

\begin{prop}
The curling number of a book graph $B_{m,n}$ is $2m$ if $n=2,3$. Moreover, the compound curling number of $B_{m,2}$ is $4m$ and that of $B_{m,3}$ is $4m^2$.
\end{prop}
\begin{proof}
The book graph $B_m$ consists of $2m$ vertices of degree $2$ and $2$ vertices of degree $m+1$. That is, the degree sequence of $B_m$ is $(2)^{2m}\circ(m+1)^2$. Hence, the curling number of $B_m$ is $2m$ and its compound curling number is  $4m$.

The graph $B_{m,3}$ consists of $2m$ vertices of degree $2$ and $m$ vertices of degree $3$, $2$ vertices of degree $m+1$ and one vertex of degree $m+2$. That is, the degree sequence of $B_{m,3}$
is $(2)^{2m}\circ (3)^m \circ (m+1)^2\circ (m+2)^1$. Hence, the curling number of $B_{m,3}$ is $2m$ and the compound curling number of $B_{m,3}$ is $4m^2$.
\end{proof}

\begin{prop}
For $n\ge 4$, the curling number of a stacked book graph $B_{m,n}=K_{1,m}\Box P_n$, $cn(B_{m,n})=m(n-2)$ and $cn^c(B_{m,n})=(2cn(B_{m,n}))^2$.
\end{prop}
\begin{proof}
In a stacked book graph $B_{m,n}$, we have $n-2$ vertices with degree $5$ (central vertices of each copy, other than the first and last, of $K_{1,m}$), $2$ vertices of degree $4$ (the central vertices of first and last copy of $K_{1,m}$), $n-2$ vertices in each copy of $P_n$ with degree $3$ and $2$ vertices in each copy of $P_n$ with degree $2$. Hence, the degree sequence $S$ of $B_{m,n}$ is given by $S=(5)^{n-2}\circ (4)^2\circ (3)^{m(n-2)}\circ (2)^{2m}$. Hence, $cn(B_{m,n})=m(n-2)$ and $cn^c(B_{m,n})=4m^2(n-2)^2=(2cn(B_{m,n}))^2$.
\end{proof}

From the above results, we also note that the the curling number of the Cartesian product of two graphs is the product of curling numbers of the factor graphs only when the length of maximal identity subsequence is sufficiently greater than that of the other subsequences. This fact highlights the scope for further studies in this direction for vast number of other known graph classes.

Invoking the above results on different graph classes, we infer following interesting result.

\begin{prop}
If $G_1(V_1,E_1)$ is a regular graph and $G_2(V_2,E_2)$ be an arbitrary graph, then $cn(G_1\Box G_2)=|V_1|\cdot cn(G_2)$. 
\end{prop}
\begin{proof}
Let $G_1$ be a $k$-regular graph on $m$ vertices and let $G_2$ be a graph on $n$ vertices. Then every vertex $v_{ij}$ in $G_1\Box G_2$ has the degree $k+d_j$, where $d_j=d(v_j); \ v_j\in V(G_2)$. Clearly, the maximal identity sequence in every copy of $G_2$ in $G_1\Box G_2$ is $(k+d_j)^{cn(G_2)}$ and hence the maximal identity sequence of degree sequence in $G_1\Box G_2$ is $(k+d_j)^{m\cdot cn(G_2)}$. Therefore, $cn(G_1\Box C_1)=m\cdot cn(G_2)$. 
\end{proof}

In the following section, we discuss the curling number and compound curling number of other products of the graph classes we discussed in the above section.

\section{Curling Number of Strong Product of Graphs}

The {\em strong product} of two graphs $G_1$ and $G_2$ is the graph, denoted by $G_1\boxtimes G_2$, whose vertex set is $V(G_1) \times V(G_2)$, the vertices $(v,u)$ and $(v',u')$ are adjacent in $G_1\boxtimes G_2$ if $[vv'\in E(G_1)~\text{and}~ u=u']$ or $[v=v' ~ \text{and}~uu'\in E(G_2)]$ or $[vv'\in E(G_1)$ and $uu'\in E(G_2)]$. 

For any vertex $(v_i,u_j)$ in the strong product graph $G_1\boxtimes G_2$, we have $d_G(v_i,u_j)=d_{G_1}(v_i)+d_{G_2}(u_j)+d_{G_1}(v_i)\,d_{G_2}(u_j)$.

We can see that the curling number of the strong product of two regular graphs and that of their Cartesian product are always equal, even though the degree of corresponding vertices in two product graphs are different. This is true for the compound curling number also. Analogous to the discussions in the previous section, the curling number of the strong product of two graphs can be determined only when the degree sequences of two factor graphs are known or predictable.  

%\begin{thm}\label{T-CNSP1}
%For two graphs $G_1$ and $G_2$, we have $cn(G_1\boxtimes G_2)=cn(G_1)cn(G_2)$.
%\end{thm}
%\begin{proof}
%Let $k_1=cn(G_1)$ and $k_2=cn(G_2)$. Let $d_i$ be the degree of $k_1$ vertices in $G_1$ and $d_j^{\prime}$ be the degree of $k_2$ vertices in $G_2$. Then, there will be $k_1k_2$ vertices in $G_1\boxtimes G_2$ having the degree $d_i+d_j^{\prime}+d_id_j^{\prime}$. Note that it is the maximum number of vertices in $G_1\boxtimes G_2$ having the same degree. Therefore, $cn(G_1\boxtimes G_2)=k_1k_2=cn(G_1)cn(G_2)$.
%\end{proof}

%Note that the above result is similar to Theorem \ref{T-CNCP1}. In view of this fact we note that the similar results hold for the strong product of regular graphs also.

Analogous results for grid graphs, prism graphs and book graphs proved in the previous section are as follows.

Since $P_2\boxtimes P_2$ is $K_4$, its curling number has already been computed. We can also verify that curling number of the  strong product graph $G=P_m\boxtimes P_n$ as follows. Since, a vertex in a path graph has degree either $1$ or $2$, then the degree of a vertex in $G$ can be $3, 5$ or $8$. It can be noted that the degree sequence $S$ of $G$ can be written as $(8)^{(m-2)(n-2)}\circ (5)^{2(m+n-4)}\circ (3)^4$. Hence, for $m,n\ge 3$, we have $cn(P_m\boxtimes P_n)=\max\{2(m+n-4),(m-2)(n-2)\}$ and  $cn^c(P_m\boxtimes P_n)=8((m-2)^2(n-2)+(m-2)(n-2)^2=cn^c(P_m\Box P_n)$.

Note that $C_m\boxtimes P_2$ is a regular graph and hence the discussion on their curling number does not offer anything new. For $m, n\ge 3$, the degree sequence of the graph $C_m\boxtimes P_n$ can be written as $(5)^{2m}\circ (8)^{m(n-2)}$. Hence, $cn(C_m\boxtimes P_n)=\max\{2m, m(n-2)\}$ and $cn^c(C_m\boxtimes P_n)=2m^2(n-2)=cn^c(C_m\Box P_n)$.  

In a similar way, the degree sequence of the graph $C_m\boxtimes C_n$ can be written as $(8)^{mn}$, as it is a regular graph. Therefore, by Proposition \ref{P-CCN}, $cn^c(C_m\boxtimes C_n)=mn=cn(C_m\boxtimes C_n)$. %In this case also, our observations hold good.

Note that $2m$ vertices of the graph $C_m\boxtimes P_2$ have degree $3$ and only two vertices have degree more than that. Therefore, $cn(K_{1,m}\boxtimes P_n)=2m$ and $cn^c(K_{1,m}\boxtimes P_n)=4m$.  For $n\ge3$, it can be noted that for the graph in the graph $G=K_{1,m}\boxtimes P_n$, analogous to a stacked book graph $B_{m,n}$, the degree of a vertex  can be $3,5, 2m+1$ or $2+3m$ and as explained above we can write its degree sequence as $S=(3)^{2m}\circ (5)^{m(n-2)}\circ (2m+1)^{2}\circ (3m+2)^{n-2}$. Therefore, $cn(K_{1,m}\boxtimes P_n))=\max\{2m,m(n-2)\}$ $cn^c(K_{1,m}\boxtimes P_n)=4m^2(n-2)^2=cn^c(B_{m,n})$.

Invoking the above observations, we notice that the Cartesian product and strong products of certain graphs, have the same curling number and the same compound curling number. But checking the degree sequences of different strong product of graphs, we can not say that this property does not hold in all cases. Therefore, the graph classes holding this property attracts more attention. 

As mentioned in the previous section, we can also observe that $cn(G_1\boxtimes G_2)=cn(G_1)\cdot cn(G_2)$ only when the lengths of maximal identity sequences of the degree sequence of the respective factor graphs are sufficiently greater than (or equal to) the lengths of other identity subsequences. Therefore, further investigation for identifying and characterising such graph classes is also promising.

%\begin{conj}\label{C-CCN}
%The compound curling number of the Cartesian product and the strong product of two graphs $G_1$ and $G_2$ are equal. That is,  for two non-empty finite graphs $G_1$ and $G_2$, $cn^c(G_1\Box G_2)=cn^c(G_1\boxtimes G_2)$. 
%\end{conj}

\section{Curling Number of Corona of Two Graphs}

The \textit{corona} of two graphs $G_1$ and $G_2$ (see \cite{RFH}), denoted by $G_1\odot G_2$, is the graph obtained by taking $|V(G_1)|$ copies of the graph $G_2$ and adding edges between each vertex of $G_1$ to every vertex of one (corresponding) copy of $G_2$. The following theorem establishes a range for the curling number of the corona of two graphs.

\begin{thm}\label{T-CNC1}
If $G_1(V_1,E_1)$ and $G_2(V_2,E_2)$ are two non-empty, non-trivial finite graphs, then $|V_1|cn(G_2)\le cn(G_1\odot G_2)\le cn(G_1)+|V_1|cn(G_2)$.
\end{thm}
\begin{proof}
Let $V_1=\{v_1,v_2,v_3\ldots,v_n\}$ and $V_2=\{u_1,u_2,u_3\ldots,u_m\}$ be the vertex sets and $S_1$ and $S_2$ be the degree sequences of $G_1$ and $G_2$ respectively. Assume that $\eta_1=cn(G_1)$ and $\eta_2=cn(G_2)$ and let $d_i$ be the base of $cn(G_1)$ in $S_1$ and $d_j^{\prime}$ be the base of $cn(G_2)$ in $S_2$. Let $G=G_1\odot G_2$. We shall denote the $i$-th copy of $G_2$ in the corona $G=G_1\odot G_2$ by $G_2^{(i)}$, whose vertex set is given by $V_2^{(i)}=\{u_j^{(i)}:u_j\in V_2, 1\le j\le m\}$, where $1\le i\le n$.

In each copy $G_2^i$, in $G$, we have $d_{G}(u_j^{(i)})=d_{G_2}(u_j)+1$ and $d_G(v_i)=d_{G_1}(v_i)+|V_2|=d_{G_1}(v_i)+m$. Hence, it can be noted that the curling number of the subgraph $H_i$ of $G$ induced by the vertex set $V_2^{(i)}\cup\{v_i\}$ is given by
\begin{equation}
cn(H_i)=
\begin{cases}
1+ \eta_2, & \text{if}~~ d_{G_1}(v_i)+m=d_j^{\prime}+1,\\
\eta_2, & \text{otherwise}.
\end{cases}
\end{equation}
Then, if no vertex $v_i$ in $V_1$ holds the condition $d_{G_1}(v_i)+m=d_j^{\prime}+1$, then $cn(G)=\sum\limits_{i=1}^n cn(H_i)=n\eta_2$. Also, note that the maximum number of vertices of $G_1$ that can hold the property $d_{G_1}(v_i)+m=d_j^{\prime}+1$ is $cn(G_1)=\eta_1$. In this case, $cn(G)=\eta_1+n\eta_2$. Therefore, $n\eta_2\le cn(G_1\odot G_2)\le \eta_1+n\eta_2$.
\end{proof}

\begin{cor}
If $G_1(V_1,E_1)$ and $G_2(V_2,E_2)$ are regular graphs, then 
$$cn(G_1\odot G_2)=
\begin{cases}
|V_1|(1+|V_2|), & \text{if} ~ d_{v\in V_1}(v)+|V_2|=1+d_{u\in V_2}(u),\\
|V_1|\,|V_2|, & \text{otherwise}.
\end{cases}$$
\end{cor}
\begin{proof}
If $d_{v_i\in V_1}(v_i)+|V_2|=1+d_{u_j\in V_2}(u_j)$, then $G_1\odot G_2$ is also a regular graph on $|V_1|(1+|V_2|)$ vertices and hence by Proposition \ref{P-CCN}, $cn(G_1\odot G_2)=|V_1|(1+|V_2|)$. 

If $d_{v_i\in V_1}(v_i)+|V_2|\ne 1+d_{u_j\in V_2}(u_j)$, then all vertices in all $V_1$ copies of $G_2$ have the same degree in $G_1\odot G_2$ which are different from the degrees of vertices of $G_1$ in $G_1\odot G_2$. Hence, $cn(G_1\odot G_2)=|V_1|\,|V_2|$. This completes the proof.
\end{proof}

The following result discusses the compound curling number of the corona of two regular graphs.

\begin{prop}
The compound curling number of the corona two regular graphs $G_1\odot G_2$ is given by $cn^c(G_1\odot G_2)=
\begin{cases}
|V_1|(|V_2|+1), & \text{if}~~ d_{G_1}(v_i)+|V_2|=1+d_{G_2}(u_j),\\
|V_1|^2 |V_2|, & \text{otherwise}.
\end{cases}$
\end{prop}
\begin{proof}
Let $\{v_1,v_2,v_3,\ldots, v_n\}$ be the vertex set of the given $r$-regular graph $G_1$ and $\{u_1,u_2,u_3,\ldots,u_m\}$ be the vertex set of the given $s$-regular graph $G_2$. Let $G=G_1\odot G_2$. Also, assume that $\{u_1^{(i)},u_2^{(i)},u_3^{(i)},\ldots,u_m^{(i)}\}$ be the vertex set of the $i$-th copy of $G_2$ in $G$. Then, $d_G(v_i)=r+m$ and $d_G(u_j^{(i)})=1+s$.

If $r+m=1+s$, then the graph $G_1\odot G_2$ becomes a regular graph on $n(m+1)$ vertices. Therefore, by Proposition \ref{P-CCN}, $cn^c(G_1\odot G_2)=n(m+1)$.

Next, assume that  $r+m=1+s$. Therefore, $G_1\odot G_2$ is not a regular graph. Since $G_1$ is a regular graph, $d(v_i)=r+m$ and hence the degree subsequence $S_1$ of $S$ corresponding to $G_1$ in $G$ is given by $S_1=(r+m)^n$. Now, let $H=\bigcup\limits_{i=1}^{n}G_2^{(i)}$. Then, for every vertex $u_j^{(i)}\in H$, $d_G(u_j^{(i)})=1+s$ and the corresponding degree subsequence $S_2$ of $S$ is    given by $(1+s)^{nm}$. Therefore, we have $S=(r+m)^n\circ (1+s)^{nm}$. Hence, $cn^c(G)=n^2m$. 
\end{proof}

\section{A Few Points on Curling Numbers of Direct Products of Graphs}

%The \textit{lexicographic product} of two graphs $G_1$ and $G_2$, denoted by $G_1[G_2]$, is the graph with the vertex set $V(G_1)\times V(G_2)$ and the vertices $(u,u^{\prime}), (v, v^{\prime})$ are adjacent in $G_1[G_2]$ if and only if either $uv\in E(G)$ or $u=v$ and $u^{\prime}v^{\prime} \in E(G_2)$. It is proved in \cite{SCML} that for a vertex $(v_i,u_j)\in G_1[G_2]$,  $d_G(v_i,u_j)=|V(G_2)|d_{G_1}(v_i)+d_{G_2}(u_j)$, where $1\le i\le |V(G_1)|, 1\le j\le |V(G_2)|$.

The \textit{tensor product} or \textit{direct product} of graphs $G_1$ and $G_2$ is the graph $G_1\times G_2$  with the vertex set $V(G_1\times G_2)=V(G_1)\times V(G_2)$ and the vertices $(v,u)$ and $(v',u')$ are adjacent in $G_1\times G_2$ if and only if $vv'\in E(G_1)$ and $uu' \in E(G_2)$. For a vertex $(v_i,u_j)\in G=G_1\times G_2$, we have $d_G(v_i,u_j)=d_{G_1}(v_i)d_{G_2}(u_j)$, where $1\le i\le |V(G_1)|, 1\le j\le |V(G_2)|$ (see \cite{SCML}). 

%\ni Analogous to  the corresponding results, we can prove the following theorem.

%\begin{thm}
%For two graphs $G_1$ and $G_2$, we have $cn(G_1[G_2])=cn(G_1\times G_2)$.
%\end{thm}
%\begin{proof}
%The proof is exactly as the proofs Theorem \ref{P-CNP1a}.
%\end{proof}

%It is also very much interesting to verify whether the above property holds for other products of two given graphs, which are not mentioned in this paper, and this problem seems to be a promising one.
Similar to that of the Cartesian product of graphs, the compatible class of direct product of the two graphs $G_1(V_1,E_1)$ and $G_2(V_2, E_2)$ can be defined as $C_k=\{(v_i,v_j)\in V_1\times V_2: d(v_i)\cdot d(v_j)=k\}$, where $\delta(G_1)\cdot \delta(G_2)\le k \le \Delta(G_1)\cdot \Delta(G_2)$. 

Determining the degree sequence of the direct product of two graphs is also possible only when the patterns of degree sequence of the factor graphs are known. In this context, we discuss the curling number and compound curling number of the direct products of graph classes we considered in the previous sections.

%\begin{prop}\label{P-CNDP}
Let $G_1$ be a $k_1$-regular graph on $m$ vertices and $G_2$ be a $k_2$-regular graphs on $n$ vertices. Then, all $mn$ vertices have the same degree $k_1\cdot k_2$. Therefore,  $cn(G_1\times G_2)=mn=cn^c(G_1\times G_2)$. 
%\end{prop}

If $G_1$ be a $k$-regular graph on $m$ vertices and $G_2$ be an arbitrary graph. Let $d(v_j)=d_j$ for $v_j\in V(G_2)$. Then, a vertex $v_{ij}$ in $V(G_1\times G_2)$ has the degree $kd_j$. Then, we have $cn(G_1\times G_2)=m\cdot cn(G_2)$.

The graph $P_2\times P_2$ is isomorphic the disjoint union $K_2\cup K_2$ and hence is $1$-regular graph on $4$ vertices. Hence, by the above result, we have $cn(P_2\times P_2)=cn^c(P_2\times P_2)=4$. Note that any graph $P_m\times P_2$ is isomorphic to the disjoint union of two paths $P_m$. The degree sequence of $P_3\times P_2$ is $(1)^4\circ (2)^2$. Therefore, $cn(P_3\times P_2)=4$ and $cn^c(P_3\times P_2)=8$. Now, for $m\ge 4$, the degree sequence of $P_m\times P_2$ is $(2)^{2(m-2)}\circ (1)^4$. Hence, for $m\ge 4$, we have $cn(P_m\times P_2)=2m-4$ and $cn^c(P_m\times P_2)=8(m-2)$. For $m,n\ge 3$, the degree sequence of $P_m\times P_n$ is $(1)^4\circ 2^{2(m+n-4)}\circ (4)^{(m-2)(n-4)}$. Therefore, $cn(P_m\times P_n)=\max\{2(m+n-4), (m-2)(n-2)\}$ and $cn^c(P_m\times P_n)=8(m-2)(n-2)(m+n-4)$.

The graph $C_m\times P_2$ is a $2$-regular graph on $2m$ vertices and hence we have $cn(C_m\times P_2)=2m$. For $n\ge 3$, the degree sequence of the graph $C_m\times P_n$ is $(2)^{2m}\circ (4)^{m(n-2)}$. Therefore, $cn(C_m\times C_n)=\{2m,m(n-2)\}$ and $cn^c(C_m\times C_n)=2m^2(n-2)$.

\section{Conclusion}

In this paper, we have discussed the concepts of curling number and compound curling number of various products of graphs. We noted an important property that the curling number of all four fundamental products of two graphs are the product of the curling numbers of the two individual graphs.  More problems regarding the curling number and compound curling number of certain other graph classes, graph operations, certain other graph products and graph powers are still to be settled. Some other problems we have identified in this area for further works are the following. 

\begin{prob}{\rm 
Determine the curling number and the compound curling numbers of different products of two graphs, in which one is a complete graph and the other is a path or a cycle.}
\end{prob}

\begin{prob}{\rm 
Determine the curling number and the compound curling number of the other graph products like lexicographic product and rooted product.}
\end{prob}

\begin{prob}{\rm 
Determine a closed formula for the curling number and the compound curling number of arbitrary products of given graphs.}
\end{prob}

\begin{prob}{\rm 
Identify and characterise the product graphs whose curling numbers are the product of the curling numbers of their factors graphs.}
\end{prob}

\begin{prob}{\rm 
Identify and characterise the graphs whose different products graphs have the same curling number (and/or the same compound curling number).}
\end{prob}

\begin{prob}{\rm 
Determine the compound curling numbers different products of graphs in which one graph is a regular graph}.
\end{prob}

There are more problems in this area which seem to be promising for further investigations. All these facts highlights a wide scope for further studies in this area.

\section{Acknowledgements}

The first and third authors of this article would like to dedicate this paper to their research supervisor Prof. (Dr.) K. A. Germina, who is currently a Professor in the Department of Mathematics, University of Botswana, Gaborone, Botswana.

\end{document}